\documentclass[12pt,reqno]{amsart}
\def\C{\mathbb C}
\def\N{\mathbb N}

\def\1{{\bf 1}}

\def\pmod #1{\ ({\rm{mod}}\ #1)}

\usepackage{amsmath, epsfig, cite}
\usepackage{amssymb}
\usepackage{amsfonts}
\usepackage{latexsym}
\usepackage{amsthm}

\newtheorem{theorem}{Theorem}[section]

\newtheorem{lemma}[theorem]{Lemma}

\theoremstyle{remark}

\numberwithin{equation}{section}

\begin{document}

\title[Two $q$-supercongruences from  Watson's transformation ]
{Two $q$-supercongruences from  Watson's transformation}

\begin{abstract}
Guo and Zudilin [Adv. Math. 346 (2019), 329--358] introduced a new
method called `creative microscoping', to prove many
$q$-supercongruences  in a unified way. In this paper, we apply this
method and Watson's ${}_8\phi_7$ transformation formula to prove two
 $q$-supercongruences, which were recently conjectured by Guo and Schlosser.
\end{abstract}
\author[He-Xia Ni]{He-Xia Ni}
\address{Department of Applied Mathematics, Nanjing Audit University\\Nanjing 211815,
People's Republic of China}
\email{nihexia@yeah.net}

\author[Li-Yuan Wang]{Li-Yuan Wang*}
\address{School of Physical and Mathematical Sciences, Nanjing Tech University , Nanjing 211816, People's Republic of China}
\email{wly@smail.nju.edu.cn}

\thanks{*Corresponding author.}
\thanks{The first author is supported by the National Natural Science Foundation of China (Grant No. 12001279).  }
\maketitle

\section{Introduction}
In 1997, Van Hamme \cite{Ha96} proposed 13 conjectured congruences concerning truncated Ramanujan-type series for $1/\pi$. For example, he made the following conjecture:
\begin{align}
\sum_{k=0}^{\frac{p-1}{2}}(4k+1)\frac{(\frac{1}{2})_k^4}{k!^4}&\equiv p\pmod{p^3},\label{DT1}
\end{align}
where $p>3$ is a prime and $(a)_n=a(a+1)\cdots (a+n-1)$ is the Pochhammer symbol. The supercongruence \eqref{DT1} was proved by
Van Hamme \cite[(C.2)]{Ha96} himself and was later shown to hold modulo $p^4$ by Long \cite{Lo}.
Moreover, applying the fact that the Calabi-Yau threefold is related to modular forms, Ahlgren and Ono \cite{AO}, Kilbourn \cite{Ki} confirmed Van Hamme's (M.2) supercongruence:
\begin{align}\label{DT3}
\sum_{k=0}^{(p-1)/2}\frac{(\frac{1}{2})_k^4}{k!^4}\equiv a_p\pmod{p^3},
\end{align}
where $a_p$ is the $p$-th coefficient of a weight $4$ modular form
\begin{align*}
\eta (2z)^4\eta(4z)^4:=q\prod_{n=1}^{\infty}(1-q^{2n})^4(1-q^{4n})^4,q=e^{2\pi iz}.
\end{align*}
In 2011, Long \cite{Lo}  made use of Whipple's ${}_7F_{6}$ transformation formula to prove that
\begin{align*}
\sum_{k=0}^{(p-1)/2}(4k+1)\frac{(\frac{1}{2})^6}{k!^6}\equiv p\sum_{k=0}^{(p-1)/2}\frac{(\frac{1}{2})_k^4}{k!^4}\pmod{p^4} \  {\rm for}\  p>3,
\end{align*}
which can be written as
\begin{align}\label{AOK}
\sum_{k=0}^{(p-1)/2}(4k+1)\frac{(\frac{1}{2})^6}{k!^6}\equiv pa_p\pmod{p^4}\ {\rm for}\  p>3
\end{align}
in light of the supercongruence \eqref{DT3}.
In 2016, Long and Ramakrishna \cite[Theorem 2]{LR} established the following supercongruence:
\begin{align}\label{DT4}
\sum_{k=0}^{p-1}(6k+1)\frac{(\frac{1}{3})_k^6}{k!^6}\equiv\begin{cases}-p\Gamma_p(1/3)^9 \pmod{p^6}, &\text{if }p\equiv 1\pmod{6},\\[5pt]
-\frac{p^4}{27}\Gamma_p(1/3)^9 \pmod{p^6},&\text{if }p\equiv 5\pmod{6},
\end{cases}
\end{align}
where $\Gamma_p(x)$ is the $p$-adic Gamma function.
Now all of Van Hamme's conjectural supercongruences have been confirmed. We refer the reader to \cite{OZ16,Sw15} for the history of the proofs of them.

During the past few years, $q$-supercongruences have been widely investigated and
a variety of techniques were involved. For more related results and the latest progress, see \cite{G2,G3,Guonew,GS,GS20,GS21,GZ1,LW,LP,NP,NP2,WY1,Zudilin2}.
In particular, Guo and Schlosser \cite[ Theorem 4.1]{GS} proposed the following partial  $q$-analogue of \eqref{AOK} and \eqref{DT4}:
for positive integers $n$ and $d$ with $d\geq 3$ and $\gcd(n,d)=1$,
\begin{align}
&\sum_{k=0}^{n-1}[2dk+1]\frac{(aq,q/a;q^d)_k(q,q^d)_k^4}{(aq^d,q^d/a;q^d)_k(q^d;q^d)_k^4}q^{(2d-3)k} \notag\\[5pt]
&\quad\equiv\begin{cases}0 \pmod{\Phi_n(q)^2}, &\text{if }n\equiv -1\pmod{d}, \\[5pt]
0 \pmod{\Phi_n(q)},&\text{otherwise}.
\end{cases}  \label{DT2}
\end{align}
Here and in what follows, we assume $q$ to be fixed with $0<|q|<1$. For $a\in \C,$ the {\it $q$-shifted factorial} \cite{GR} is defined by
$$
(x;q)_n=\begin{cases}
(1-x)(1-xq)\cdots(1-xq^{n-1}), &\text{if }n\geq 1,\\[5pt]
1, &\text{if }n=0,
\end{cases}
$$
and the {\it $n$-th cyclotomic polynomial} is defined as
$$
\Phi_n(q):=\prod_{\substack{1\leq k\leq n\\ (n,k)=1}}(q-e^{2\pi\sqrt{-1}\cdot\frac{k}{n}}).
$$
Also, we frequently use the shortened notation:
$$(a_1,\ldots, a_m;q)_k=(a_1;q)_k\cdots (a_m;q)_k,\ k\in \N\cup \infty.$$

The main purpose of this paper is to prove the following results, which were recently conjectured by
Guo and Schlosser \cite[Conjectures 4.2 and 5.11]{GS}.
\begin{theorem} \label{Conject1}
Let $n$ and $d$ be positive integers with $d\geq 3$ and $\gcd(n,d)=1$. Then
\begin{align}\label{more5}
&\sum_{k=0}^{n-1}[2dk+1]\frac{(aq,q/a,bq,q/b;q^d)_k(q;q^d)_k^2}{(aq^d,q^d/a,bq^d,q^d/b;q^d)_k(q^d;q^d)_k^2}q^{(2d-3)k}\notag\\
&\quad\equiv\begin{cases}0 \pmod{[n]\Phi_n(q)}, &\text{if }n\equiv -1\pmod{d},\\[5pt]
0 \pmod{[n]},&\text{otherwise}.
\end{cases}
\end{align}
\end{theorem}

\begin{theorem}\label{Conject2}
Let $n>1$ and $d\geqslant 3$ be integers with $\gcd(n,d)=1$. Then
\begin{align}\label{more6}
&\sum_{k=0}^{n-1}[2dk-1]\frac{(aq^{-1},q^{-1}/a,bq^{-1},q^{-1}/b;q^d)_k(q^{-1};q^d)_k^2}{(aq^d,q^d/a,bq^d,q^d/b;q^d)_k(q^d;q^d)_k^2}q^{(2d+3)k}\notag\\
&\quad\equiv\begin{cases}0 \pmod{[n]\Phi_n(q)}, &\text{if }n\equiv 1\pmod{d},\\[5pt]
0 \pmod{[n]},&\text{otherwise}.
\end{cases}
\end{align}
\end{theorem}

It is clear that \eqref{more5} is a two-parameters generation of \eqref{DT2}.

Following Gasper and Rahman \cite{GR}, the ${}_{r+1}\phi_r$ basic hypergeometric series is defined by
\begin{equation*}
{}_{r+1}\phi_r\!\left[\begin{matrix}
a_1,a_2,\dots,a_{r+1}\\b_1,\dots,b_r
\end{matrix};q,z\right]:=\sum_{k=0}^\infty
\frac{(a_1,a_2,\dots,a_{r+1};q)_k}{(q,b_1,\dots,b_r;q)_k}
z^k.
\end{equation*}
In the proof of Theorems \ref{Conject1} and \ref{Conject2}, we will make use of Waston's ${}_8\phi_7$ transformation formula \cite[Appendix (III.17)]{GR}:
\begin{align}\label{Wasttran}
&{}_8\phi_7\bigg[\begin{matrix}a, &qa^{\frac{1}{2}},&-qa^{\frac{1}{2}},&b,&c,&d,&e,&f\\ &a^{\frac{1}{2}},&-a^{\frac{1}{2}},&aq/b,&aq/c,&aq/d,&aq/e,&aq/f\end{matrix};\ q,\ \frac{a^2q^2}{bcdef}\bigg]\notag\\
&\quad=\frac{(aq,aq/de,aq/df,aq/ef;q)_\infty}{(aq/d,aq/e,aq/f,aq/def;q)_\infty}{}_4\phi_3\bigg[\begin{matrix}aq/{bc},&d,&e,&f\\ &aq/b,&aq/c,&def/a\end{matrix};\ q,\ q\bigg],
\end{align}
which is valid whenever the ${}_8\phi_7$ series converges and the ${}_4\phi_3$ series terminates.

\section{Proof of Theorems \ref{Conject1} and \ref{Conject2}}

We need the following lemma, which was proved by Guo and Schlosser \cite[Lemma 2.1]{GS20}.
\begin{lemma}\label{Th1proofLemma1}
Let $m,n$ and $d$ be positive integers with $m\leq n-1$. Let $r$ be an integer satisfying $dm\equiv -r\pmod{n}$. Then, for $0\leq k\leq m,$ we have
\begin{align*}
\frac{(aq^r;q^d)_{m-k}}{(q^d/a;q^d)_{m-k}}\equiv (-a)^{m-2k}\frac{(aq^r;q^d)_k}{(q^d/a;q^d)_k}q^{m(dm-d+2r)/2+(d-r)k}\pmod{\Phi_{n}(q)}.
\end{align*}
\end{lemma}

In order to prove Theorems \ref{Conject1} and \ref{Conject2}, we also need to establish the following three-parametric $q$-congruences.
\begin{lemma}\label{Th1proofLemma2}
Let $n$ and $d$ be positive integers with $\gcd(n,d)=1$. Let $r$ be an integer and let $a,b,c $ be indeterminates. Then
\begin{align}
&\sum_{k=0}^{m}[2dk+r]\frac{(aq^r,q^r/a,bq^r,q^r/b,q^r/c,q^r;q^d)_k}{(aq^d,q^d/a,bq^d,q^d/b,cq^d,q^d;q^d)_k}(cq^{2d-3r})^k \equiv 0\pmod{[n]},\label{super3}\\[5pt]
&\sum_{k=0}^{n-1}[2dk+r]\frac{(aq^r,q^r/a,bq^r,q^r/b,q^r/c,q^r;q^d)_k}{(aq^d,q^d/a,bq^d,q^d/b,cq^d,q^d;q^d)_k}(cq^{2d-3r})^k \equiv 0\pmod{[n]},\label{super4}
\end{align}
where $0\leq m\leq n-1$ and $dm\equiv -r\pmod{n}.$
\end{lemma}

\begin{proof} It is easy to see that Lemma \ref{Th1proofLemma2} is true  for $n=1$ or $r=0$. We now suppose that $n>1$ and $r\neq 0$. By Lemma \ref{Th1proofLemma1}, for $0\leq k\leq m,$ the $k$-th and $(m-k)$-th terms on the left-hand side of \eqref{super1} cancel each other modulo $\Phi_n(q)$, i.e.,
\begin{align*}
&[2d(m-k)+r]\frac{(aq^r,q^r/a,bq^r,q^r/b,q^r/c,q^r;q^d)_{m-k}}{(aq^d,q^d/a,bq^d,q^d/b,cq^d,q^d;q^d)_{m-k}}(cq^{2d-3r})^{m-k}\notag\\
&\quad\equiv-[2dk+r]\frac{(aq^r,q^r/a,bq^r,q^r/b,q^r/c,q^r;q^d)_k}{(aq^d,q^d/a,bq^d,q^d/b,cq^d,q^d;q^d)_k}(cq^{2d-3r})^k\pmod{\Phi_n(q)}.
\end{align*}
This proves that the $q$-congruence \eqref{super3} is true modulo $\Phi_n(q).$

Furthermore, the $q$-factorial $(q^r;q^d)_k$ has a factor of the form $1-q^{\alpha n}$ for $m<k\leq n-1,$
and so it is congruent to $0$ modulo $\Phi_{n}(q)$ as $dm\equiv-r \pmod{n}$. Note that $(q^d;q^d)_k$ and $\Phi_{n}(q)$ are coprime for $m<k\leq n-1$. Thus, the $k$-th summand in \eqref{super4} with $k$ satisfying $m<k\leq n-1$ is congruent to $0$ modulo $\Phi_{n}(q)$.
This together with \eqref{super3} modulo $\Phi_{n}(q)$ confirms the $q$-congruence \eqref{super4} modulo $\Phi_{n}(q)$.

We are now ready to prove \eqref{super3} and \eqref{super4} modulo $[n]$.
Let $\zeta\neq 1$ be a primitive root of unity of degree $s$ with $s|n$ and $s>1$. Let $c_q(k)$ be the $k$-th term on the left-hand side of \eqref{super3}, i.e.,
$$
c_q(k)=[2dk+r]\frac{(aq^r,q^r/a,bq^r,q^r/b,q^r/c,q^r;q^d)_k}{(aq^d,q^d/a,bq^d,q^d/b,cq^d,q^d;q^d)_k}(cq^{2d-3r})^k.
$$
The $q$-congruences \eqref{super3} and \eqref{super4} modulo $\Phi_{n}(q)$ with $n\mapsto s$ indicate that
$$\sum_{k=0}^{m_1}c_\zeta(k)=\sum_{k=0}^{s-1}c_\zeta(k)=0,$$ where $dm_1\equiv -r\pmod{s}$ and $0\leq m_1 \leq s-1.$
Observe that
$$\lim_{q\rightarrow \zeta}\frac{c_q(ls+k)}{c_q(ls)}=\frac{c_\zeta(k)}{[r]}.$$
It follows that
$$\sum_{k=0}^{n-1}c_\zeta(k)=\sum_{l=0}^{n/s-1}\sum_{k=0}^{s-1}c_\zeta(ls+k)=\frac{1}{[r]}\sum_{l=0}^{n/s-1}c_\zeta(ls)\sum_{k=0}^{s-1}c_\zeta(k)=0,$$
and
$$\sum_{k=0}^{m}c_\zeta(k)=\frac{1}{[r]}\sum_{l=0}^{(m-m_1)/s-1}c_\zeta(ls)\sum_{k=0}^{s-1}c(k)+\frac{c_\zeta(m-m_1)}{[r]}\sum_{k=0}^{m_1}c_\zeta(k)=0,$$
which imply that both $\sum_{k=0}^{n-1}c_q(k)$ and $\sum_{k=0}^{m}c_q(k)$ are congruent to $0$ modulo $\Phi_s(q)$.
The proof then follows the fact that $\prod_{s\mid n,s>1}\Phi_{s}(q)=[n].$
\end{proof}

\begin{theorem}\label{Th1}
Let $a,b $  be indeterminates and $r=\pm 1$. Let $n$ and $d$ be integers satisfying $d\geq 3$ and $n>1,$ such that $\gcd(n,d)=1$, and $n\equiv -r\pmod{d}$.  Then
\begin{align}
&\sum_{k=0}^{M}[2dk+r]\frac{(aq^r,q^r/a,bq^r,q^r/b,q^r,q^r;q^d)_k}{(aq^d,q^d/a,bq^d,q^d/b,q^d,q^d;q^d)_k}q^{(2d-3r)k} \equiv 0\pmod{[n]\Phi_{n}(q)},\label{super1}
\end{align}
where $ M=(dn-n-r)/d$ or $n-1.$
\end{theorem}
\begin{proof} Since $n>1$ and $r=\pm 1,$ we see that
$$
(dn-n-r)/d\leq (dn-n+1)/d\leq n-1.
$$
Letting $q\rightarrow q^d$ and taking $a=q^r, b=aq^r, c=q^r/a, d=bq^r, e=q^r/b, f=q^{r-dn+n}$ in \eqref{Wasttran}, we obtain
\begin{align*}
&\sum_{k=0}^{M}[2dk+r]\frac{(aq^r,q^r/a,bq^r,q^r/b,q^{r-dn+n},q^r;q^d)_k}{(aq^d,q^d/a,bq^d,q^d/b,q^{d+dn-n},q^d;q^d)_k}q^{(2d-3r+dn-n)k} \\
&\quad=[dn-n]\frac{(q^r,q^{d-r};q^d)_{(dn-n-r)/d}}{(q^d/b,bq^{d};q^d)_{(dn-n-r)/d}}\sum_{k=0}^{M}\frac{(q^{d-r},bq^{r},q^{r}/b,q^{r-dn+n};q^d)_kq^{dk}}{(q^d/a,aq^d,q^{2r-dn+n},q^d;q^d)_k}.
\end{align*}
Namely,
\begin{align}\label{section21}
&\sum_{k=0}^{M}[2dk+r]\frac{(aq^r,q^r/a,bq^r,q^r/b,q^r/c,q^r;q^d)_k}{(aq^d,q^d/a,bq^d,q^d/b,cq^{d},q^d;q^d)_k}(cq^{2d-3r})^k \notag\\
&\quad\equiv [dn-n]\frac{(q^r,q^{d-r};q^d)_{(dn-n-r)/d}}{(q^d/b,bq^{d};q^d)_{(dn-n-r)/d}} \notag\\ &\quad\times\sum_{k=0}^{M}\frac{(q^{d-r},bq^{r},q^{r}/b,q^{r}/c;q^d)_kq^{dk}}{(q^d/a,aq^d,q^{2r}/c,q^d;q^d)_k} \quad\pmod{(c-q^{dn-n})}.
\end{align}

On the other hand, by Lemma \ref{Th1proofLemma2}, the left-hand side of \eqref{section21} is congruent to 0 modulo $[n]$.
Moreover, $[dn-n]$ is also congruent to $0$ modulo $\Phi_n(q)$, and therefore \eqref{section21} also holds modulo $\Phi_n(q)$. Since the polynomials $\Phi_n(q)$ and $c-q^{dn-n}$ are coprime, we have
\begin{align}\label{super5}
&\sum_{k=0}^{M}[2dk+r]\frac{(aq^r,q^r/a,bq^r,q^r/b,q^r/c,q^r;q^d)_k}{(aq^d,q^d/a,bq^d,q^d/b,cq^d,q^d;q^d)_k}(cq^{2d-3r})^k \notag\\
&\quad\equiv[dn-n]\frac{(q^r,q^{d-r};q^d)_{(dn-n-r)/d}}{(q^d/b,bq^{d};q^d)_{(dn-n-r)/d}}\notag\\
&\quad\times\sum_{k=0}^{M}\frac{(q^{d-r},bq^{r},q^{r}/b,q^{r}/c;q^d)_kq^{dk}}{(q^d/a,aq^d,q^{2r}/c,q^d;q^d)_k}\quad\pmod{(c-q^{dn-n})\Phi_n(q)},
\end{align}

For any integer $x$, let $f_k(x)$ be the least non-negative integer $k$ such that $(q^x;q^d)_k\equiv 0\pmod{\Phi_{n}(q)}$. Since $n\equiv -r\pmod{d},$ we have $f_k(d-r)=(n+r)/d, f_k(r)=(d(n+1)-(n+r))/d, f_k(d)=n, f_k(2r)=(d(n+1)-2(n+r))/d$.
It is easy to see that
$$ f_k(d) \geq f_k(r)=(d(n+1)-(n+r))/d>(dn-n-r)/d\geq f_k(2r)\geq f_k(d-r) $$
for $r= \pm 1$ and $d\geq 3$.
Hence, when $c\rightarrow 1$, the denominator of the reduced form of the $k$-th summand
$$\frac{(q^{d-r},bq^r,q^r/b,q^r;q^d)_k}{(q^d/a,aq^d,q^{2r},q^d;q^d)_k}q^{dk}$$
in the ${}_4\phi_3$ summation is always relatively prime to $\Phi_{n}(q)$ for $k\geq 0$.
This proves that the congruences \eqref{super1} hold modulo $\Phi_{n}(q)^2$ by noticing that $(q^{d-r};q^d)_{(dn-n-r)/d}$ contains the factor $1-q^n$ and $q^n\equiv 1 \pmod{\Phi_{n}(q)}$. Using Lemma \ref{Th1proofLemma2} and ${\rm lcm}(\Phi_{n}(q)^2,[n])=[n]\Phi_n(q)$, the proof of the theorem is complete.
\end{proof}

\begin{proof}[Proof of Theorems \ref{Conject1} and \ref{Conject2} ]
Taking $c=1$ in \eqref{super3} and \eqref{super4} and applying Lemma \ref {Th1proofLemma2}, we see that the $q$-congruences \eqref{more5} and \eqref{more6} hold modulo $[n]$.
Further, the proof of the modulus $[n]\Phi_n(q)$ case of \eqref{more5} and \eqref{more6} then follows from  Theorem \ref{Th1}.
\end{proof}

\vskip 5mm \noindent{\bf Acknowledgment.}
%The authors are grateful to the anonymous referee for valuable comments that helped to improve the quality of the article.
We thank Professor Victor J. W. Guo  for his  helpful comments on this paper.

\end{document}